\documentclass{amsart}

\usepackage{amssymb}
\usepackage{amsthm}
\usepackage{amsmath} 
\usepackage{amsbsy}
\usepackage[all]{xy}
\usepackage{bm}
\usepackage{hyperref}
\usepackage{tikz}
\usepackage{cite}
\newtheorem*{theorem}{Theorem}

\newtheorem*{corollary}{Corollary}

\newtheorem{lemma}{Lemma}

\newcommand{\abs}[1]{\lvert #1 \rvert}
\newcommand{\norm}[1]{\lVert #1 \rVert}

\begin{document}

\title[]{Optimal Trapping for Brownian motion: a nonlinear analogue of the torsion function} 

\thanks{The research of J.L. was supported in part by the National Science Foundation under award DMS-1454939. S.S. was partially supported by the NSF (DMS-1763179) and the Alfred P. Sloan Foundation. J.L. would like to thank Jian Ding and James Nolen for helpful discussion.}

\author[]{Jianfeng Lu}
\address[Jianfeng Lu]{Department of Mathematics, Department of Physics, and Department of Chemistry,
Duke University, Box 90320, Durham NC 27708, USA}
\email{jianfeng@math.duke.edu}

\author[]{Stefan Steinerberger}
\address[Stefan Steinerberger]{Department of Mathematics, Yale University, New Haven, CT 06510, USA}
\email{stefan.steinerberger@yale.edu}

\begin{abstract} We study the problem of maximizing the expected lifetime of drift diffusion in a bounded domain. More
formally, we consider the PDE
\[ - \Delta u + b(x) \cdot \nabla u = 1 \qquad \mbox{in}~\Omega\]
subject to Dirichlet boundary conditions for $\|b\|_{L^{\infty}}$ fixed. We show that, in any given $C^2-$domain $\Omega$, the vector field maximizing the expected lifetime is (nonlinearly) coupled to the solution and satisfies $b = -\|b\|_{L^{\infty}} \nabla u/ \abs{\nabla u}$ which reduces the problem to the study of the nonlinear PDE
\[ -\Delta u - b \cdot \left| \nabla u \right| = 1,\]
where $b = \|b\|_{L^{\infty}}$ is a constant.
We believe that this PDE is a natural and interesting nonlinear analogue of the torsion function. We prove that, for fixed volume,  $\| \nabla u\|_{L^1}$ and $\|\Delta u\|_{L^1}$ are maximized if $\Omega$ is the ball (the ball is also known to maximize $\|u\|_{L^p}$ for $p \geq 1$ from a result of Hamel \& Russ).
\end{abstract}
\maketitle

\vspace{-10pt}

\section{Introduction}

We consider, for open and bounded $\Omega \subset \mathbb{R}^d$, solutions of the equation
\begin{align*}
 -\Delta u+ b(x) \cdot \nabla u &= 1  \qquad \text{in }\Omega \\
u &= 0 \qquad \text{on }\partial \Omega.
\end{align*}

This equation arises naturally as the expected lifetime of a drift-diffusion
$$ \mathrm{d} X_t = b(X_t) \,\mathrm{d} t + \sqrt{2} \,\mathrm{d} B_t,$$
where $B$ is standard Brownian motion and $b:\Omega \rightarrow \mathbb{R}^d$ is a vector field. 

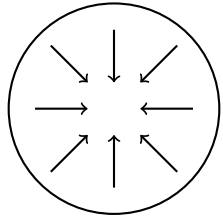
\begin{figure}[h!]
\centering
\begin{tikzpicture}[scale=0.7]
\draw [thick] (0,0) circle (2cm);
\draw [thick, ->] (-1.5,0) -- (-0.5,0);
\draw [thick, ->] (1.5,0) -- (0.5,0);
\draw [thick, ->] (0,-1.5) -- (0,-0.5);
\draw [thick, ->] (0,1.5) -- (0,0.5);
\draw [thick, ->] (1.2, 1.2) -- (0.5,0.5);
\draw [thick, ->] (-1.2, 1.2) -- (-0.5,0.5);
\draw [thick, ->] (1.2, -1.2) -- (0.5,-0.5);
\draw [thick, ->] (-1.2, -1.2) -- (-0.5,-0.5);
\end{tikzpicture}
\caption{Brownian motion stays trapped for the longest time if it moves inside a ball and the vector field pushes it radially inside.}
\end{figure}

Our main question is the following: for what vector field $b$ (fixing its maximal strenght $\|b\|_{L^{\infty}}$) ncan we maximize the expected lifetime of Brownian motion?  It is clear that allowing for a stronger vector field $\|b\|_{L^{\infty}}$ increases our ability to trap
the particle. It is not terribly difficult to see
that, given $\|b\|_{L^{\infty}}$ and $|\Omega|$, the quantity $\|u\|_{L^{p}}$ is finite and can be controlled in terms of those
two parameters and the dimension,  however, we are interested in the sharp dependence.

\section{The Result}
\subsection{Main result.} We now state our main result.

\begin{theorem} Among all bounded $C^2-$domains $\Omega$ with fixed volume and all vector fields $b:\Omega \rightarrow \mathbb{R}^d$ with  $\|b\|_{L^{\infty}}$  fixed, the solution of
\begin{align*}
 -\Delta u+ b \cdot \nabla u &= 1  \qquad \text{in }\Omega \\
u &= 0 \qquad \text{on }\partial \Omega
\end{align*}
maximizes
$$\int_{\Omega}{ |\nabla u| dx}, \qquad \mbox{and} \qquad \int_{\Omega}{ -\Delta u dx},$$
when $\Omega$ is the ball and $b = -\|b\|_{L^{\infty}} \nabla u/|\nabla u|$. 
\end{theorem}
  It is clear from the proof that the result is optimal
up to possibly the regularity conditions on the boundary of $\Omega$:
having an irregular boundary should make it more difficult to
effectively trap Brownian motion and one could thus expect that it is
possible to slightly weaken the assumption.  Our proof is based on
first showing that the vector field
$b = -\|b\|_{L^{\infty}} \nabla u/|\nabla u|$ is the best choice in
any domain $\Omega$ -- this nonlinear condition results in the
(mildly) nonlinear PDE
$$ -\Delta u - \|b\|_{L^{\infty}} |\nabla u| = 1.$$
This PDE has one notable property: it is invariant under adding constants. In particular, if $u$ is a solution to the equation on $\Omega$ with Dirichlet boundary conditions, then $u - \varepsilon$ is a solution to the PDE on the domain
$$ \Omega_{\varepsilon} = \left\{x \in \Omega: u(x) \geq \varepsilon\right\}.$$
This allows an elementary induction over level sets.
\begin{corollary}[also implied by Hamel \& Russ \cite{hamel2}] Under the same assumption,
$$ \|u\|_{L^p(\Omega)} \qquad \mbox{for}~p \in \mathbb{N}_{\geq 1}~\mbox{is maximized by the ball.}$$
\end{corollary}
 We emphasize that the Corollary is not new and known at a greater level of generality, for $1 \leq p \leq \infty$, from a very general rearrangement principle of Hamel \& Russ \cite{hamel2}. However, our proof is very different and gives a particularly elementary derivation for the case $p=\infty$.

\subsection{Existing results.} The case $b \equiv 0$ is classical. P\'olya \cite{pol} proved that the integral over the solution of $-\Delta u = 1$ increases under symmetrization. The statement for the $L^{\infty}-$norm follows from a now classical theorem of Talenti \cite{talenti}. We also refer to Ba\~nuelos \& Carroll~\cite{ban} and Burchard \& Schmuckenschl\"ager~\cite{almut}. The solution of $-\Delta u = 1$ has been studied in great detail, see e.g., \cite{banc, banc2, beck, lustein, makar, stein1, volkov}; we also refer to the textbooks of Baernstein \cite{al}, Bandle \cite{bandle} and P\'olya-Szeg\H{o} \cite{polsz} for more details about the case $b \equiv 0$. There is a general rearrangement inequality due to Hamel \& Russ \cite{hamel2} that can be applied to general semi-elliptic equations of the type
$$ -\mbox{div}(a(x) \cdot \nabla u) + h(x, u, |\nabla u|) = f(x)$$
which implies the corollary for general $1 \leq p \leq \infty$.

\subsection{Broader outlook.} We believe that the partial differential equation
$$ \boxed{ -\Delta u - b \cdot |\nabla u| = 1}$$
with Dirichlet boundary conditions may be of broader interest. It is a classical and very difficult problem to study the level sets of solutions of elliptic PDEs   \cite{banc, banc2, beck, beck2, brascamp, brasco, b4, freitas, grieser, hamel, kawohl, magna, makar, manas, stein1, talenti}. An example of a basic question \cite{lions} is whether solutions in convex domains `inherit' the convexity of the domain and have convex level sets; this was shown to hold for the solution of $-\Delta u =1$ by Makar-Limanov \cite{makar} and for the first Laplacian eigenfunction $-\Delta u = \lambda_1 u$ by Brascamp-Lieb \cite{brascamp} but is known to fail \cite{hamel} for the general equation $-\Delta u = f(u)$.\\

The equation $ -\Delta u - b \cdot |\nabla u| = 1$ shares
many characteristics with the torsion function $-\Delta u = 1$ and is
perhaps its simplest nonlinear analogue. In particular, it is not very
difficult to show that for $b \rightarrow 0$ it
converges to the torsion function (and thus has convex level sets on
convex domains); conversely, for
$b \rightarrow \infty$, the interpretation as a
drift-diffusion suggests that the solution should be of the form
$\ln u(x) \asymp \tfrac{1}{2}b \cdot \mbox{dist}(x, \partial
\Omega)$ (see e.g. \cite{fw}) and should also have convex level sets on convex domains;
one could wonder whether this is then also true in the intermediate
regime $b = 1$. There are several other results about level sets \cite{beck, makar, stein1, volkov} that may be interpreted as a stepping stones to a more complete theory of level sets of elliptic PDEs, we believe that  $ -\Delta u - b \cdot |\nabla u| = 1$ might be another natural test case.

\section{The Proof}
\subsection{An Application of the Maximum Principle.} We first establish that the optimal vector field is nonlinearly coupled to the solution via
$$b = -\|b\|_{L^{\infty}} \frac{\nabla u}{|\nabla u|}.$$
We actually show a stronger result saying that for any solution $u$, replacing the vector field by $b = -\|b\|_{L^{\infty}} \nabla u/|\nabla u|$ increases the function everywhere.

\begin{lemma}
 Suppose
\begin{align*}
  -\Delta w + b \cdot \nabla w = 1 \qquad \mbox{in}~\Omega
 \end{align*}
 with Dirichlet boundary conditions. Then, with the convention that $\nabla w/|\nabla w| =  0$ whenever $\nabla w = 0$,
 the solution of
 \begin{align*} 
 -\Delta u - \|b\|_{L^{\infty}} \frac{\nabla w}{|\nabla w|} \cdot \nabla u = 1   \qquad \mbox{in}~\Omega
\end{align*}
with Dirichlet boundary conditions satisfies
$$ u \geq w.$$
\end{lemma}

\begin{proof}
We observe that whenever $\nabla w \neq 0$, by Cauchy-Schwarz,
$$  b \cdot \nabla w \geq - \|b\|_{L^{\infty}} \abs{\nabla w} = - \norm{b}_{L^{\infty}} \frac{\nabla w}{|\nabla w|} \cdot \nabla w.$$
Recalling our convention that  $\nabla w/|\nabla w| =  0$ whenever $\nabla w = 0$, this inequality continues to hold in that case as well.
Thus
$$  -\Delta w - \|b\|_{L^{\infty}} \frac{\nabla w}{|\nabla w|} \cdot \nabla w \leq 1.$$
Subtracting the two solutions yields
$$ -\Delta (u-w) -  \|b\|_{L^{\infty}} \frac{\nabla w}{|\nabla w|} \cdot (\nabla u - \nabla w) \geq 0.$$
The maximum principle now implies
\begin{equation*}
u \geq w.\qedhere
\end{equation*}
\end{proof}

This Lemma reduces the problem to the study of the nonlinear PDE
$$  -\Delta u - \|b\|_{L^{\infty}}| \nabla u| = 1.$$  
Lemma 1 has an interesting geometric interpretation (see Fig. 2):
suppose we have a given vector field $b$ that gives rise to a profile
of exit times $u$. Let us consider a small neighborhood around a point
$u(x_0)$. In order to ensure that the diffusion particle survives for
a longer time, we would like the force field $b$ to push it in a
suitable direction. However, the suitable direction is given by $u$
itself: larger values of $u$ mean larger lifetime, we want to locally
push the particle in direction $\nabla u(x_0)$. It is this geometric
interpretation that suggests that the PDE might perhaps be considered
a rather natural nonlinear analogue of the torsion function
$-\Delta u = 1$.

\begin{figure}[h!]
\centering
\begin{tikzpicture}[scale=1]
\draw [ultra thick] (0,0) to[out=320, in = 180] (3,-1);
\draw [dashed] (0,0.3) to[out=320, in = 180] (3,-1+0.3);
\draw [dashed] (0,0.7) to[out=320, in = 180] (3,-1+0.7);
\draw [dashed] (0,1) to[out=320, in = 180] (3,-1+1);
\draw [thick, ->] (1.5, -0.8) -- (1.7, 0);
\draw [thick, ->] (1, -0.7) -- (1.3, 0);
\end{tikzpicture}
\caption{Lemma 1 illustrated: the optimal vector field is orthogonal to the level sets.}
\end{figure}
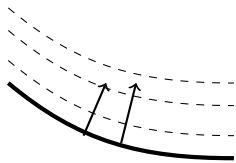

\subsection{An estimate on the $L^1$ norm of the gradient.} The purpose of this section is to establish part of the statement of the main Theorem: 
among all domains $\Omega$ with fixed volume, we are interested in solutions of
$$ -\Delta u - \|b\|_{L^{\infty}} |\nabla u| = 1$$
with Dirichlet boundary conditions. Among those solutions
$$ \int_{\Omega}{ | \nabla u| dx} \qquad \mbox{is maximized by the ball.}$$
This is the main result of this section. We abbreviate, for the remainder of the argument, $b = \|b\|_{L^{\infty}}$.
Before discussing the main argument of the section, we argue that the inwards pointing normal derivative cannot vanish.

\begin{lemma} 
Let $\Omega \subset \mathbb{R}^d$ be a bounded $C^2$ domain. Then, for some constant $c_{\Omega} > 0$ and all $b \geq 0$, the solution of
$$ -\Delta u - b |\nabla u| = 1$$
satisfies
$$ u(x) \geq c_{\Omega} \cdot d(x, \partial \Omega),$$
where $d(x, \partial \Omega)$ is the distance to the boundary. In particular, the normal derivative does not vanish on the boundary.
\end{lemma}
\begin{proof}
The result is known to hold for all positive, super-harmonic functions in bounded $C^2-$domains. It is known as the Zaremba-Hopf-Oleinik Lemma or, sometimes, as boundary point lemma; we refer to Kuran \cite{kuran}, Nazarov \cite{nazarov} or the book of Pucci \& Serrin \cite{pucci}. The solution of
$$ -\Delta w = 1$$
with Dirichlet boundary conditions 
is such a positive, superharmonic function in $\Omega$ and thus satisfies the inequality. Moreover, by Lemma 1, we have $u \geq w$ and this implies the result.
\end{proof}

Lemma 2 can be extended to slightly rougher domains (which is not the focus of our paper). It is known that a $C^{1,1}$ condition suffices and there has been work on finding the exact threshold of regularity that is required for the boundary point lemma to apply, see for example Apushkinskaya \& Nazarov \cite{apu}. \\

We can now state the main result of this section. For simplicity of exposition, we
 define, for $b > 0$ a fixed constant, the function
$$ f(c) = b \cdot \sup_{|\Omega| = c} ~~\int_{\Omega} { |\nabla u| dx}.$$
We observe that, using the PDE and a Green formula,
\begin{align*}
f(c) =\sup_{|\Omega| = c} \int_{\Omega} {b \cdot |\nabla u| dx}  &=\sup_{|\Omega| = c} \left(  \int_{ \Omega} -\Delta u - 1 ~dx\right)\\
&=\sup_{|\Omega| = c} \left(  - | \Omega| + \int_{\partial \Omega} \frac{\partial u}{\partial n}~ dx\right) \\
&= - c + \sup_{|\Omega| = c}\int_{\partial \Omega} \frac{\partial u}{\partial n}~ dx
\end{align*}
where $n$ is the inward pointing normal vector. So we can write equivalently
$$ f(c) + c = \sup_{|\Omega| =c} ~~\int_{\partial \Omega}~ \frac{\partial u}{\partial n} dx.$$
We introduce one last constant $c_d$ as the sharp constant in the isoperimetric inequality in the formulation
$$ | \partial \Omega| \geq c_d |\Omega|^{\frac{d-1}{d}}.$$
The main result of this section is the following estimate.
\begin{lemma} The function $f$ satisfies the differential inequality
$$ f'(c) \leq  b\frac{f(c) +c}{c_d c^{\frac{d-1}{d}}}.$$
\end{lemma}
\begin{proof}
We start by noting the elementary estimate, using $\partial u/\partial n > 0$,
\begin{align*}
 | \partial \Omega| = \int_{\partial \Omega} 1 &\leq \left( \int_{\partial \Omega}{ \frac{1}{\frac{\partial u}{\partial n}} dx} \right)^{1/2}  \left( \int_{\partial \Omega}{ \frac{\partial u}{\partial n} dx} \right)^{1/2} \\
&\leq  \left( \int_{\partial \Omega}{ \frac{1}{\frac{\partial u}{\partial n}} dx} \right)^{1/2} (f(c)+c)^{1/2}
 \end{align*}
 and therefore
 \begin{equation} \label{ineq1}
  \int_{\partial \Omega}{ \frac{1}{\frac{\partial u}{\partial n}} dx} \geq \frac{|\partial \Omega|^2}{f(c) + c}.
 \end{equation}
We note that
$$ f(c) =  \sup_{|\Omega| = c} ~~\int_{\Omega} { b \cdot |\nabla u| dx}$$
is invariant under subtracting constants. We can thus introduce $\Omega_{\varepsilon}$ as the region
where $u-\varepsilon$ is positive (and note that $u-\varepsilon$ satisfies the nonlinear PDE with Dirichlet boundary conditions on $\Omega_{\varepsilon}$). Then
\begin{align*}
 \int_{\Omega} {b \cdot |\nabla u| dx} &= b\int_{u \leq \varepsilon} { |\nabla u| dx} + b\int_{\Omega_{\varepsilon}} { |\nabla (u-\varepsilon)| dx}\\
 &\leq  b\int_{u \leq \varepsilon} { |\nabla u| dx} + f(|\Omega_{\varepsilon}|).
 \end{align*}
The coarea formula shows that
$$ \int_{u \leq \varepsilon} { |\nabla u| dx} = \int_{0}^{\varepsilon} \mathcal{H}^{d-1} \left\{ x: u(x) = t\right\} dt,$$
where $\mathcal{H}^{d-1}$ is the $(d-1)-$dimensional Hausdorff measure.
Since the normal derivative does not vanish on the boundary and since the boundary is $C^2$, we have, as $\varepsilon \rightarrow 0$,
$$ \int_{u \leq \varepsilon} { |\nabla u| dx} = \varepsilon | \partial \Omega| + o(\varepsilon).$$
In particular, this shows that $f(c)$ is continuous in $c$, as
$$f(\abs{\Omega}) - f(\abs{\Omega_{\varepsilon}}) \leq \varepsilon
\abs{\partial \Omega} + o(\varepsilon).$$
  We also observe that, using
\eqref{ineq1} and that the solution is $u \in C^2$ (see e.g. \cite{caff})
\begin{align*}
 | \Omega_{\varepsilon}| &= |\Omega| - \varepsilon  \int_{\partial \Omega}{ \frac{1}{\frac{\partial u}{\partial n}} dx}  + o(\varepsilon)\\
 & \leq |\Omega| - \varepsilon \frac{| \partial \Omega|^2}{f(c)+c} + o(\varepsilon).
\end{align*}
Therefore, for $\varepsilon \rightarrow 0^+$,
\begin{align*}
b\int_{\Omega}{|\nabla u| dx} &\leq   b\int_{u \leq \varepsilon} { |\nabla u| dx} + b\int_{\Omega_{\varepsilon}} { |\nabla (u-\varepsilon)| dx} \\
&\leq b\varepsilon |\partial \Omega| + f(|\Omega_{\varepsilon}|) + o(\varepsilon)\\
&\leq  b \varepsilon |\partial \Omega| + f\left( c -  \frac{| \partial \Omega|^2}{f(c) + c} \varepsilon\right) + o(\varepsilon), 
\end{align*}
where we have used the continuity of $f$. 
Rearranging and letting $\varepsilon \rightarrow 0$ then implies
$$ f'(c) \leq b\frac{f(c) +c}{|\partial \Omega|}$$
and the desired result follows from an application of the isoperimetric inequality.
\end{proof}

\medskip 
We now argue that Lemma 3 is optimal for the ball. This is an explicit computation that we will now carry out.
 Let us define
 $$ h(c) + c = \int_{\partial B_c}~ \frac{\partial u}{\partial n} dx,$$
where $B_c$ is the ball normalized to satisfy $|B_c| = c$.
\begin{lemma}
We have
$$ h'(c) =  b\frac{h(c) +c}{|\partial B_c|}.$$
\end{lemma}
We observe that $ | \partial B_c| = c_d |B_c|^{\frac{d-1}{d}} = c_d c^{\frac{d-1}{d}}$, so the ODE coincides exactly with the upper bound derived in Lemma 3.
\begin{proof} The PDE 
$$ -\Delta u - b |\nabla u| = 1$$
has a radial solution on the ball. Moreover, the solution is monotonically decreasing. Assuming the ball is centered at the origin, we can rewrite the Laplacian in polar coordinates and obtain the ODE for $g(\abs{x}) : = u(x)$: 
\begin{equation} \label{maineq}
- \frac{1}{r^{d-1}} \frac{\partial}{\partial r} \left( r^{d-1} \frac{\partial g}{\partial r}\right) + b g'(r) = 1. 
\end{equation}
A priori we would be forced to solve the ODE again and again for balls
of different volume setting Dirichlet boundary condition: this is not
the case for this particular ODE since we have invariance under adding
constants. In particular, we may fix arbitrary initial conditions, say
$g(0) = 0$. On a ball with radius $R$, the solution is then given by
$g(r) - g(R)$.  We observe that the expression $r^{d-1} g'(r)$
corresponds exactly to the normal derivative. More formally, we note
that
$$ h'(c) = -1 + \frac{\partial}{\partial c} \int_{\partial B_c}~ \frac{\partial u}{\partial n} dx,$$
where the derivative with respect to $c$ is with respect to volume. Let us denote the ball with radius $R$ by $B^R$ and let us assume that
$R$ is chosen such that $|B^R|=c$. Then, locally, around $R$, we have
$$ |B^{R+\varepsilon}| \sim |B^R| + \varepsilon |\partial B^R|$$
and this suggests the change of variables
$$ \frac{\partial}{\partial c} \int_{\partial B_c}~ \frac{\partial u}{\partial n} dx = \frac{1}{|\partial B^R|}\frac{\partial}{\partial r} \int_{\partial B^r}~ \frac{\partial u}{\partial n} dx ~\big|_{r=R}.$$
This is where we can use the equation (\ref{maineq}) which, after multiplying with the normalizing volume and rearranging, looks like
$$-  \frac{\partial}{\partial r} \left(\omega_d r^{d-1} \frac{\partial g}{\partial r}\right) = \omega_d r^{d-1} - b \omega_d g'(r) r^{d-1}.$$
We observe that 
$$ \frac{\partial}{\partial r} \int_{\partial B^r}~ \frac{\partial u}{\partial n} dx = -  \frac{\partial}{\partial r} \left(\omega_d r^{d-1} \frac{\partial g}{\partial r}\right)$$
and thus
$$ \frac{1}{|\partial B^R|}\frac{\partial}{\partial r} \int_{\partial B^r}~ \frac{\partial u}{\partial n} dx ~\big|_{r=R} = \frac{ \omega_d R^{d-1}}{|\partial B^R|}  - \frac{b g'(R) \omega_dR^{d-1}}{|\partial B^R|}.$$
We note that, by definition,
$$  \frac{ \omega_d R^{d-1}}{|\partial B^R|}  = 1$$
and
\begin{align*}
-\frac{b g'(R) \omega_dR^{d-1}}{|\partial B^R|} &= \frac{b}{|\partial B^R|}  \int_{\partial B^R}~ \frac{\partial u}{\partial n} dx\\
&=\frac{b}{|\partial B_c|}\left( h(c) + c\right)
\end{align*}
which is the desired statement.
\end{proof}

\subsection{An Estimate for the $L^{\infty}-$norm.} A similar argument, coupled with our estimate on $f(c)$, can be used to show the main result.
We define
$$ g(c) = \sup_{|\Omega|=c}~ \|u\|_{L^{\infty}}.$$
\begin{lemma} We have
$$ g'(c) \leq \frac{f(c) +c}{c_d^2 c^{ \frac{2d-2}{d}}}$$
with equality if and only if the domain is a ball.
\end{lemma}
\begin{proof} As before, given any domain $\Omega$, we can consider the domain $\Omega_{\varepsilon}$ on which $(u - \varepsilon)_{+}$ is positive (and thus solves the PDE there). For $\varepsilon$ sufficiently small, this domain $\Omega_{\varepsilon}$ satisfies, as above,
\begin{align*}
 | \Omega_{\varepsilon}| &= |\Omega| - \varepsilon  \int_{\partial \Omega}{ \frac{1}{\frac{\partial u}{\partial n}} dx}  + o(\varepsilon)\\
 &\leq |\Omega| - \varepsilon \frac{| \partial \Omega|^2}{f(c)+c} + o(\varepsilon).
  \end{align*}
  Moreover, we have the elementary fact that
  $$ \| u \|_{L^{\infty}(\Omega)} = \varepsilon +  \| u - \varepsilon \|_{L^{\infty}(\Omega_{\varepsilon})}.$$
This implies that, for $\varepsilon$ sufficiently small,
$$ g(c) \leq \varepsilon + g\left( c -  \frac{| \partial \Omega|^2}{f(c)+c} \varepsilon\right) + \mathcal{O}(\varepsilon^2)$$
and thus
$$ g'(c) \leq \frac{f(c)+c}{|\partial \Omega|^2} \leq \frac{f(c) +c}{c_d^2 c^{\frac{2d-2}{d}}}.$$
However, $f$ is maximized for the ball.  Conversely, if we are dealing
with the ball, then the normal derivative $\partial u/\partial n$ is
constant on the boundary and we have equality in the bound
\begin{align*}
 | \partial \Omega| = \int_{\partial \Omega} 1 &= \left( \int_{\partial \Omega}{ \frac{1}{\frac{\partial u}{\partial n}} dx} \right)^{1/2}  \left( \int_{\partial \Omega}{ \frac{\partial u}{\partial n} dx} \right)^{1/2} \\
&=  \left( \int_{\partial \Omega}{ \frac{1}{\frac{\partial u}{\partial n}} dx} \right)^{1/2} (f(c)+c)^{1/2}.
 \end{align*}
This then implies equality in the bound
\begin{align*}
 | \Omega_{\varepsilon}| &= |\Omega| - \varepsilon  \int_{\partial \Omega}{ \frac{1}{\frac{\partial u}{\partial n}} dx}  + o(\varepsilon)\\
 & = |\Omega| - \varepsilon \frac{| \partial \Omega|^2}{f(c)+c} + o(\varepsilon).
\end{align*}
Altogether, this then implies that we have equality in the bound for $g'(c)$ and this shows that we have equality for the ball.
\end{proof}

\subsection{An Estimate for the $L^{p}-$norm.} We conclude by adapting the argument to the $L^p$-norm. We argue similarly and introduce the function
$$ h_p(c) = \sup_{|\Omega|=c}~ \|u\|_{L^{p}}^p$$
and will again argue starting at level set $\varepsilon$, calling the arising superlevel set $\Omega_{\varepsilon}$. 
\begin{lemma} We have, for all $p \geq 1$, $p \in \mathbb{N}$
$$ h_p'(c) \leq p\cdot h_{p-1}(|\Omega|) \cdot \frac{f(c)  + c}{|\partial \Omega|^2}.$$
\end{lemma}
\begin{proof}
We decompose
\begin{align*}
 \int_{\Omega} {u^p dx} &= \int_{u \leq \varepsilon} { u^p dx} + \int_{\Omega_{\varepsilon}} {u^p dx}\\
&=  \int_{u \leq \varepsilon} { u^p dx} + \int_{\Omega_{\varepsilon}} {(u_{\varepsilon} + \varepsilon)^p dx}.
\end{align*}
The first term is fairly easy to deal with since, as $\varepsilon \rightarrow 0^+$, we have
$$ \int_{u \leq \varepsilon} { u^p dx} = (1+o(1)) \frac{\varepsilon^{p+1}}{p+1} \int_{\partial \Omega} \left( \frac{\partial u}{\partial n} \right)^p d\sigma = o(\varepsilon).$$
The second term can be expanded, asymptotically, like
$$ \int_{\Omega_{\varepsilon}} {(u_{\varepsilon} + \varepsilon)^p dx} = \int_{\Omega_{\varepsilon}} {u_{\varepsilon}^p dx} + p \varepsilon  \int_{\Omega_{\varepsilon}} {u_{\varepsilon}^{p-1} dx} + o(\varepsilon).$$ 
We now argue first in the case of $p=1$. We obtain
$$ h_1(|\Omega|) \leq \varepsilon |\Omega_{\varepsilon}| + h_1\left( | \Omega_{\varepsilon}|\right)$$
and use the inequality 
$$ | \Omega_{\varepsilon}|  \leq |\Omega| - \varepsilon \frac{| \partial \Omega|^2}{f(c)+c} + o(\varepsilon)$$
to argue that
$$ h_1(|\Omega|) \leq \varepsilon \left( |\Omega| - \varepsilon \frac{| \partial \Omega|^2}{f(c)+c} \right) + h_1\left( |\Omega| - \varepsilon \frac{| \partial \Omega|^2}{f(c)+c}\right) + o(\varepsilon).$$
Letting $\varepsilon \rightarrow 0$, we obtain
$$ h_1(|\Omega|) \leq \varepsilon  |\Omega|  + h_1\left( |\Omega| - \varepsilon \frac{| \partial \Omega|^2}{f(c)+c}\right) + o(\varepsilon)$$
and thus
$$ h_1'(c) \leq \frac{c(f(c) + c)}{|\partial \Omega|^2}.$$
We obtain, as before, equality in the case of the ball. This settles the case $p=1$. We will now bootstrap this estimate to higher values of $p$. Arguing as above, we obtain
$$ h_p(|\Omega|) \leq p \cdot \varepsilon \cdot h_{p-1}( |\Omega_{\varepsilon}|) + h_p\left( | \Omega_{\varepsilon}|\right)$$
and this results in the inequality
$$ h_p'(c) \leq p\cdot h_{p-1}(|\Omega|) \cdot \frac{f(c)  + c}{|\partial \Omega|^2}.$$
Again, we have equality for the ball. This establishes the desired statement for $\|u\|_{L^p}$ and $p \in \mathbb{N}$.
\end{proof}

\end{document}